\documentclass[a4paper,10pt]{article}
\usepackage[T1,T2A]{fontenc}
\usepackage[utf8]{inputenc}
\usepackage{color}
\usepackage[unicode]{hyperref}
\hypersetup{colorlinks,bookmarksopen,bookmarksnumbered,citecolor=blue,pdfstartview=FitB}
\usepackage{xcolor}
\hypersetup{
    colorlinks,
    linkcolor={red!50!black},
    citecolor={blue!50!black},
    urlcolor={blue!80!black}
}
\usepackage{amsfonts,amssymb,amsmath,amscd,amsthm}
\usepackage{mathrsfs} 

\makeatletter
\def\@settitle{\begin{center}%
    \baselineskip14\p@\relax
    \bfseries
    \@title
  \end{center}%
}
\makeatother

\usepackage{longtable}

\numberwithin{equation}{section}

\newtheorem{theorem}[equation]{Theorem}
\newtheorem*{theorem*}{Theorem}
\newtheorem{proposition}[equation]{Proposition}
\newtheorem*{proposition*}{Proposition}
\newtheorem{problem}[equation]{Problem}

\newtheorem{lemma}[equation]{Lemma}
\newtheorem*{lemma*}{Lemma}
\newtheorem{corollary}[equation]{Corollary}
\newtheorem*{corollary*}{Corollary}

\theoremstyle{definition}
\newtheorem{example}[equation]{Example}

\newtheorem{definition}[equation]{Definition}

\theoremstyle{remark}
\newtheorem{remark}[equation]{Remark}
\makeatletter
\newcommand{\symbitem}[1]{\item[#1]%
\renewcommand{\@currentlabel}{#1}\ignorespaces}
\makeatother
\newcommand{\beq}{\begin{equation}}
\newcommand{\eeq}{\end{equation}}
\newcommand{\beqa}{\begin{eqnarray}}
\newcommand{\eeqa}{\end{eqnarray}}
\newcommand{\beaa}{\begin{eqnarray*}}
\newcommand{\ben}{\begin{eqnarray*}}
\newcommand{\eaa}{\end{eqnarray*}}
\newcommand{\een}{\end{eqnarray*}}

\def \G {\mathcal{G}}

\def \L {\mathcal{L}}

\def \N {\mathcal{N}}
\def \O {\mathcal{O}}

\def \S {\mathcal{S}}

\def \A {\mathbb{A}}
\def \C {\mathbb{C}}

\def \P {\mathbb{P}}
\def \Q {\mathbb{Q}}
\def \R {\mathbb{R}}

\def \Z {\mathbb{Z}}

\def \ge {\geqslant}
\def \geq {\geqslant}

\def \leq {\leqslant}

\def \kappa {\varkappa}

\def\={\;=\;}
\def\bal{\begin{aligned}}
\def\eal{\end{aligned}}

\newcommand{\Spec}{{\text{Spec }}}

\newcommand{\udot}{{\:\raisebox{3pt}{\text{\circle*{1.5}}}}}
\def \bullet {\udot}

\DeclareMathOperator{\rk}{rk}

\DeclareMathOperator{\Pic}{Pic}

\providecommand{\arxiv}[1]{\href{http://arxiv.org/abs/#1}{arXiv:#1}}

\author{Sergey Galkin}
\begin{title}
{Fano--Mathieu correspondence}
\end{title}
\date{Preprint IPMU 10-0150. First version: 2010. Last version: \today}
\begin{document}
\maketitle
\abstract{
We show that $G$-Fano threefolds are \emph{mirror-modular}.

1. \emph{Mirror maps} are inversed reversed \emph{Hauptmoduln} for moonshine subgroups of $SL_2(\mathbb{R})$.

2. \emph{Quantum periods}, shifted by an integer constant (eigenvalue of quantum operator on primitive cohomology)
are expansions of weight $2$ modular forms (``theta-functions'')
in terms of inversed Hauptmoduln.

3. Products of inversed Hauptmoduln with some fractional powers of shifted quantum periods
are very nice cuspforms (``eta-quotients'').

The latter cuspforms also appear in work of Mason and others: they are eta-products, related to conjugacy classes of sporadic
simple groups, such as Mathieu group $M_{24}$ and Conway's group of isometries of Leech lattice.

This gives a strange correspondence
between deformation classes of $G$-Fano threefolds
and conjugacy classes of Mathieu group $M_{24}$.
}
\setcounter{section}{0}
\section*{Introduction} \label{intro}
The simplest example of Lian--Yau's \emph{mirror moonshine} for $K3$ surfaces (\cite{LY1, LY2}, see also \cite{Dol, Dor, VY,Yui})
 is the remarkable identity (\emph{modular relation}) of Kachru--Vafa \cite{KV}, that goes back to Fricke and Klein:
\begin{equation} \label{kv}
\sum_{n\geq 0} \frac{(6n)!}{(3n)! n!^3} j(q)^{-n} =  E_4(q)^{\frac12}  
\end{equation}
expressing modular form $E_4(q)$ of weight $4$ 
as square of hypergeometric series $I_1(t) = \sum \frac{(6n)!}{(3n)! n!^3} t^n$
expanded in terms of inverse modular parameter $t=j(q)^{-1}$.
Here $\eta(q) = q^{\frac{1}{24}} \prod_{n\geq 1} (1-q^n) $ is Dedekind's eta-function,
$\Delta(q) = \eta(q)^{24}$ is modular discriminant,
$\sigma_3(n) = \sum_{d|n} d^3$, 
Eisenstein series $E_4(q) = 1 + 240 \sum_{n\geq 1} \sigma_3(n) q^n$ 
equals to the theta-series $\theta_{E_8}$ for a lattice $E_8$, 
and $j(q) = \frac{E_4^3}{\Delta(q)}
 = \frac{1}{q} + 744 + 196884 q + \dots$ is Klein's modular $j$-invariant.
The respective weights of these modular forms are: $wt([\eta,\Delta,E_4,j]) = [\frac12,12,4,0]$.
Multiplying both sides by $j(q)^{-\frac16}$, and then taking to the power six implies an identity for $\Delta(q)$:
\begin{equation} \label{delta}
j(q)^{-1} \left(\sum_{n\geq 0} \frac{(6n)!}{(3n)! n!^3} j(q)^{-n}\right)^6 = 
\left(\sum_{n\geq 0} \frac{(6n)!}{(3n)! n!^3} j(q)^{-n-\frac16}\right)^6 =  \eta(q)^{24} = \Delta(q)
\end{equation}
One advantage of formula \ref{delta} over \ref{kv} is that $\Delta$ is a holomorphic cuspform.
Moreover, it is an eigenform for Hecke operations, and has no zeroes on upper half-plane.
That is, $\Delta$ is an \emph{eta-product}.

Smooth anticanonical divisor $S \in | -K_Y |$  in Fano threefold $Y$ is a $K3$ surface endowed with a natural
lattice polarization $c_1(Y) \in \Pic(Y) \subset \Pic(S) \subset H^2(S,\Z) = {II}_{3,19}$.
Beauville \cite{Be} shows inverse: generic $K3$ surface lattice-polarized by $c_1(Y) \in Pic(Y)$
is anti-canonically embedded to a Fano threefold $Y$.
So Fano threefolds single out $105$ (see \cite{IP,MM2}) out of countably many families of lattice polarized $K3$ surfaces
(and also $105$ mirror dual families \cite{Dol}).
In fact, almost all moonshine examples listed in \cite{LY1, LY2, Dol, VY, Yui} 
appear in that way,
and mirror moonshine makes more sense in the context of Fano threefolds.
Golyshev \cite{Go1} reproduced 
Iskovskikh's 
classification of prime
Fano threefolds (see e.g. \cite{IP})
by effectively combining three elements: mirror, moonshine and minimality.

{\bf $G$-series} $G_Y$ (see \ref{def-gseries}) is a certain 
invariant of Fano variety $Y$ ``counting'' rational curves on it. 

{\bf Mirror} conjecture (for variations of Hodge structures) states that Laplace transform of $G$-series for Fano threefold $Y$
is the solution of Picard--Fuchs differential equation for some $1$-parameter family of $K3$ surfaces that is called \emph{mirror dual to $Y$ 
Ginzburg--Landau model}.

{\bf Moonshine}  (\emph{genus-zero modularity}) is explicitly stated as
\emph{miraculous eta-product formula}:
\begin{equation} \label{mepf}
 I_{N,s}(H_{N,c}^{-1}) = \eta(q)^2 \eta(q^N)^2 H_{N,c}^{\frac{N+1}{12}}
\end{equation}
where
$G_{Y_N}$ is $G$-series of Fano threefold $Y_N$ with invariants $H^2(Y_N,\Z) = \Z c_1(Y_N)$, $c_1(Y_N)^3 = 2N$,
$I_{N,s} = \mathrm{R}_s G_{Y_N}(t)$ is Laplace transform of $G_{Y_N}$ multiplied by $e^{s t}$ 
for a particular choice of constant $s = s_N$ (see \ref{sr},\ref{def-iseries}),
and $H_N$ is a Hauptmodul on Fricke modular curve $X_0(N)/w_N$ (see \cite{CN})
with a particular constant term $c = c_N$: $H_{N,c} = \frac{1}{q} + c_N + O(q)$.
In case $Y$ is a sextic double solid
(i.e. a smooth sextic hypersurface in weighted projective $4$-space $\P(1,1,1,1,3)$)
we have $N=1$, $H_1 = j(q)$, $s_1=120$, $c_1 = 744$ and formula \ref{mepf} specializes to \ref{kv} 
(we present other exact equalities in appendix \ref{golmod}).
One of the points of this note, is that one also has more natural identity
for a cusp-eigenform (eta-product or eta-quotient), that specializes to \ref{delta},
and the latter identity generalizes to a more general class of $G$-Fano threefolds,
with eta-products usually enoded by conjugacy classes of Mathieu group $M_{24}$.

{\bf Minimality} is formalized in the notion of \emph{$D3$ differential equation} (see \ref{myd3}).
It is a $6$-parameter class of differential equations of degree $3$,
generalizing the construction of regularized quantum differential equations
of a Fano threefold from $6$ two-point Gromov--Witten invariants. 

{\bf Modularity conjecture} (which is now a theorem) states that function $G = \sum_{n \geq 0} a_i t^n$ is
$G$-series of minimal Fano threefold of index one if and only if for some $s$ function $\R_s G$
is of moonshine type (satisfy \ref{mepf} for some $N$) 
and is annihilated by differential equation of type $D3$.

Apart from $17$ quantum differential equations of minimal smooth Fano threefolds
Golyshev 
found two more differential equations of type $D3$ and modular origin (\ref{ode14} and \ref{ode15})
(also these equations were found by Almkvist, van Enckevort, van Straten and  Zudilin in \cite{AESZ},
and they were already listed in \cite{LY2}).

It turned out (\cite{Ga3}) that these two examples are quantum differential equations for two deformation classes $Y_{28}$
and $Y_{30}$ of Fano threefolds with $H^2(Y_{28},\Z) = \Z^2$, $H^2(Y_{30},\Z) = \Z^3$.

Since Fano threefolds $Y_{28}$ and $Y_{30}$ are not minimal 
one naively expects their regularized quantum differential equations to be of
degree $4$ and $5$, but these varieties occur to be \emph{quantum minimal}  ---
minimal differential equation vanishing $\widehat{G}$-series of these varieties has degree $3$ (see \cite{Ga4} for the details).

In \cite{Ga3} we made an observation that both $Y_{28}$ and $Y_{30}$ are \emph{$G$-Fano threefolds}
i.e. for some complex structure they admit a finite group action $G:Y$ with $Pic^G(Y) = \Z c_1(Y)$.
In \cite{Ga4} we shown that $G$-Fano threefolds are quantum minimal.
So it is natural to look whether other $G$-Fano threefolds are mirror-modular.

Families $Y_{28}$ and $Y_{30}$ are two of total eight families of $G$-Fano threefolds (see \cite{Pro}).
In these article we will show that all $8$ families of $G$-Fano threefolds are mirror-modular. 




\section{Preliminaries} \label{preliminary}

\subsection{Shifts and regularizations.} \label{sr}
For a number $s$ and a power series $A = \sum_{n \geq 0} a_i t^n$
define its regularization (Laplace transform $\L$),
inverse Laplace transform $\L^{-1}$,
shifted regularization $\L_s$,
regular shift $\S_g$
and normalization $\N$ by the formulas
\begin{gather*}
\widehat{A} = \L A  = \sum (a_i \cdot n!) t^n, \\
\L^{-1} A = \sum \frac{a_i}{n!} t^n, \\
\L_s{A} = \L (e^{s \cdot t} \cdot A) \\
\S_s A = \L_s \L^{-1} A  \\
\N A = S_{-a_1} A
\end{gather*}

\subsection{Fano varieties}[see e.g. \cite{IP}]

Let $Y$ be a Fano variety --- smooth variety with ample anticanonical class $\omega_Y^{-1}$.

$Y$ is simply-connected, by Kodaira vanishing $H^i(Y,\O_Y) = 0$ for $i>0$,
 so $c_1 : \Pic(Y) \to H^2(Y,\Z)$ is an isomorphism and both are isomorphic to $\Z^\rho$,
where $\rho$ is called Picard number. Lefschetz pairing on $H^2(Y,\Z)$ defined by
$(A,B) = \int_{[Y]} A \cup B \cup c_1(Y)^{\dim Y - 2}$ is nondegenerate (by Hard Lefschetz theorem),
so $H^2(Y,\Z)$ is a lattice and we denote its discriminant by $d(Y)$.

Anticanonical \emph{degree} of Fano variety $Y$ is
$deg(Y) = (c_1(Y),c_1(Y)) =\int_{[Y]} c_1(Y)^{\dim Y} $.
Euler number is a topological Euler characteristic $\chi(Y) = \int_{[Y]} c_{\dim Y}$.

\emph{Fano index} $r(Y)$ is divisibility of $c_1(Y)$ in the lattice $H^2(Y,\Z)$
i.e. $r(Y) = \max \{ r\in \Z | c_1(X) = r H, H \in H^2(Y,\Z) \}$.

\begin{definition}[\cite{Pro}]
Fano variety $Y$ with group action $G:Y$ is called $G$-Fano if $H^2(Y,\Z)^G = \Z$.
\end{definition}

\subsection{Quantum differential equations, G-series and Givental's constant}
Let $\star$ be quantum multiplication on cohomologies of $Y$
defined by
\begin{equation}
\int_{[Y]} (\gamma_1 \star \gamma_2) \cup \gamma_3 = \sum_{d \geq 0} \left<\gamma_1, \gamma_2, \gamma_3\right>_d t^d
\end{equation}
where closed genus $0$ $3$-point correlator
 $\left<\gamma_1, \gamma_2, \gamma_3\right>_d = \int_{\mathcal{\overline{M}}_{0,3}(Y,d)} \prod ev_i^*(\gamma_i)$ is a Gromov-Witten invariant ``counting''
maps $f: \P^1 \to Y$ of degree $d = \int_{[\P^1]} f^* c_1(Y)$ passing through homology classes Poincare-dual to $\gamma_i$.

Take $D = t \frac{d}{dt}$ and define \emph{quantum differential equation}
as a trivial vector bundle over $\Spec \C[t,t^{-1}]$ with fibre $H^*(Y)$ and connection
\begin{equation} \label{qde}
D \Phi = c_1(Y) \star \Phi
\end{equation}
where $\Phi \in H^*(Y) [[t]]$.

Let $\G_Y(t) = [pt] + \sum_{n\geq1} \G^{(n)} t^n$ be the unique analytic solution of \ref{qde} starting with class Poincare-dual to the the class of the point, and define $G$-series as
\begin{equation} \label{def-gseries}
 G_Y (t) = \int_{[Y]} \G_Y (t) = 1 + \sum_{n\geq 1} G^{(n)} t^n
\end{equation}

We note that the first coefficient $G^{(1)} = \left<[pt], c_1(Y), [Y]\right>_1 = \int_{|t|=\epsilon} (\int_{[Y]} c_1(Y) \star [pt]) \frac{dt}{t^2}$ 
is zero according to the String equation,
and we name $G^{(2)} = \left< [pt] \right>_2$ (the expected number of anticanonical conics passing through a point) 
as \emph{Givental's constant} $G(Y)$, so
\begin{equation}
G_Y = 1 + G(Y) \cdot t^2 + O(t^3).
\end{equation}

Define 
$\widehat{G}$-series (also known as \emph{the quantum period} of $Y$) as 
\begin{equation} \label{def-rgseries}
 \widehat{G}_Y = \widehat{G_Y} = 1 + \sum_{n \geq 1} n! \cdot G^{(n)} t^n = 1 + 2 G(Y) \cdot t^2 + \dots
\end{equation}
Conjecturally $\widehat{G}_Y$ should have integer coefficients.

For convenience we define $I$-series to be a regularized shifted $G$-series:
\begin{definition} \label{def-iseries}
For a given number $s$ define $I_{Y,s} = \L_s G_Y$, in particular $I_{Y,0} = \widehat{G}_Y$.
By a slight abuse of notation we will say that power series $I (t) = 1 + \sum_{n \geq 1} i^{(n)} t^n$
 is a regular $I$-series of smooth Fano variety $Y$ 
if $\N I = \widehat{G}_Y$ i.e. $I = I_{Y,i^{(1)}}$.
\end{definition}

\section{G-Fano threefolds and A-model G-series} \label{geometric}

There are $8$ deformation classes of $G$-Fano threefolds $Y$ with $\rk H^2(Y, \Z) > 1$ (see e.g. \cite{Pro}).
Two of them has Fano index two, these are 
$Y_{48}^{(3)} = \P^1 \times \P^1 \times \P^1$ and $Y_{48}^{(2)} = W \subset \P^2\times\P^2$ of degree $48$.
Six other have Fano index one:
$Y_{30}$ of degree $30$,  
$Y_{28}$ of degree $28$,
$Y_{24}$ of degree $24$,
$Y_{20}$ of degree $20$,
$Y_{12}^{(2)}$ and $Y_{12}^{(3)}$ of degree $12$.

In this section we are going to describe all of them geometrically.
The details regarding the computation of the respective quantum periods
can be found in \cite{CCGK}.

\begin{definition}
Fano threefold $Y_{48}^{(2)}$ is the variety $Fl(1,2,3)$ of complete flags in $\P^2$
i.e. a hyperplane section of Segre fourfold $\P^2 \times \P^2 \subset \P^8$.
This variety is unique in its deformation class number $32$ in table $2$ of \cite{MM2},
it has Fano index $2$, degree $48$, $\chi=6$ and $\rho=2$.
\end{definition}
This variety can also be described as a projectivization of tangent bundle on $\P^2$ or variety of complete flags in $\P^2$.

\begin{corollary}
$I$-series $I_{6,2;2} = \R G_{Fl(1,2,3)}$ is given by the pullback of hypergeometric series from two-dimensional torus
\begin{equation}
I_{6,2;2} = \sum_{a,b \geq 0} \frac{(a+b)! (2a+2b)!}{a!^3b!^3} t^{2(a+b)} = 
 1 + 4 t^2 + 60 t^4 + 1120 t^6 + 24220 t^8 + 567504 t^{10} + \dots
\end{equation}
\end{corollary}
\begin{proof}
Combine the computation of $I$-series of toric variety $\P^2 \times \P^2$ in \cite{Gi97}
and quantum Lefschetz principle in \cite{CG}.
\end{proof}

\begin{definition}
Fano threefold $Y_{48}^{(3)}$ is just a Cartesian cube of a line i.e. Segre threefold $\P^1 \times \P^1 \times \P^1 \subset \P^7$.
This variety is unique in its deformation class number $27$ in table $3$ of \cite{MM2},
it has Fano index $2$, degree $48$, $\chi=8$ and $\rho=3$.
\end{definition}

\begin{corollary}
$I$-series $I_{6,3;2} = \R G_{\P^1 \times \P^1 \times \P^1}$ is given by the pullback of hypergeometric series from three-dimensional torus
\begin{equation}
I_{6,3;2} = \sum_{a,b,c \geq 0} \frac{(2a+2b+2c)!}{a!^2 b!^2 c!^2} t^{2(a+b+c)} 
= 1 + 6 t^2 + 90 t^4 + 1860 t^6 + 44730 t^8 + 1172556 t^{10} + \dots
\end{equation}
\end{corollary}

\begin{remark}
Threefolds $Y_{48}^{(2)}$ and $Y_{48}^{(3)}$ has isomorphic hyperplane sections --- del Pezzo surface of degree $6$.
This implies the relation
$G_{\P^1 \times \P^1 \times \P^1} (\sqrt{t}) = G_{Fl(1,2,3)} (\sqrt{t}) \cdot e^t$
\end{remark}


\begin{definition}
Fano threefold $Y_{30}$ is the blowup of a curve of bidegree $(2,2)$ on $Fl(1,2,3) \subset \P^2 \times \P^2$.
This deformation class of varieties has number $13$ in table $3$ of \cite{MM2},
it has degree $30$, $\chi=8$ and $\rho=3$.
\end{definition}

\begin{proposition}
$Y_{30}$ is a complete intersection of three numerically effective divisors
of tridegrees $(1,1,0)$, $(1,0,1)$ and $(0,1,1)$ on $\P^2 \times \P^2 \times \P^2$.
\end{proposition}

\begin{corollary}
$I$-series $I_{15} = \L_3 G_{Y_{30}}$ is given by the pulback of hypergeometric series from three-dimensional torus
\begin{equation}
I_{15} = \sum_{a,b,c \geq 0} \frac{(a+b)!(a+c)!(b+c)!(a+b+c)!}{a!^3 b!^3 c!^3} t^{a+b+c}
= 1 + 3 t + 15 t^2 + 105 t^3 + 855 t^4 + 7533 t^5 + \dots
\end{equation}
\end{corollary}


Let $Q$ be $3$-dimensional quadric.

\begin{definition}
Let $Y_{28}$ be the blowup of a twisted quartic on $Q$.
This deformation class of varieties has number $21$ in table $2$ of \cite{MM2},
it has degree $28$, $\chi=6$ and $\rho=2$.
\end{definition}

Denote the $I$-series of $Y_{28}$ as $I_{14} = \widehat{G}_{Y_{28}}$.


\begin{definition}
Fano threefold $Y_{24}$ is a hyperplane section of Segre embedding
$\P^1 \times \P^1 \times \P^1 \times \P^1 \subset \P^{15}$.
This deformation class of varieties has number $1$ in table $4$ of \cite{MM2},
it has degree $24$, $\chi=6$ and $\rho=4$.
\end{definition}

\begin{corollary}
$I$-series $I_{12} = \L_4 G_{Y_{24}}$ is given by the pullback of hypergeometric series from four-dimensional torus
\begin{equation}
I_{12} = \sum_{a,b,c,d \geq 0} \frac{(a+b+c+d)!^2}{a!^2 b!^2 c!^2 d!^2} t^{a+b+c+d} = 
 1 + 4 t + 28 t^2 + 256 t^3 + 2716 t^4 + 31504 t^5 + \dots
\end{equation}
\end{corollary}


\begin{definition}
Let $Y_{20}$ be the blowup of projective space $\P^3$ with
center a curve of degree $6$ and genus $3$ which is an intersection of cubics.
This deformation class of varieties has number $12$ in table $2$ of \cite{MM2},
it has degree $20$, $\chi=0$ and $\rho=2$.
\end{definition}

\begin{proposition}
$Y_{20}$ is an intersection of Segre variety $\P^3 \times \P^3$ by linear subspace of codimension $3$.
\end{proposition}

\begin{corollary}
$I$-series $I_{10} = \L_2 G_{Y_{20}}$ is given by the pulback of hypergeometric series from two-dimensional torus
\begin{equation}
I_{20} = \sum_{a,b \geq 0} \frac{(a+b)!^4}{a!^4 b!^4} t^{a+b} = 
1 + 2 t + 18 t^2 + 164 t^3 + 1810 t^4 + 21252 t^5 + 263844 t^6 + 3395016 t^7  + \dots
\end{equation}
\end{corollary}

\begin{definition}
Fano threefold $Y^{(2)}_{12}$ can be described either as a section of Segre fourfold $\P^2 \times \P^2 \subset \P^8$ by quadric
or as double cover of $W$ with branch locus in anticanonical divisor.
This deformation class of varieties has number $6$ in table $2$ of \cite{MM2},
it has degree $12$, $\chi=-12$ and $\rho=2$.
\end{definition}

\begin{corollary}
$I$-series $I_{6,2} = \L_4 G_{Y^{(2)}_{12}}$ is given by the pullback of hypergeometric series from two-dimensional torus
\begin{equation}
I_{6,2} = \sum_{a,b \geq 0} \frac{(a+b)! (2a+2b)!}{a!^3b!^3} t^{a+b} 
= 1 + 4 t + 60 t^2 + 1120 t^3 + 24220 t^4 + 567504 t^5 + \dots
\end{equation}
\end{corollary}

\begin{remark}
Series $I_{6,2}$ and $I_{6,2;2}$ are related by change of coordinates $I_{6,2;2} (t) = I_{6,2} (t^2)$.
\end{remark}

\begin{definition}
Fano threefold $Y^{(3)}_{12}$ is a double cover of $\P^1 \times \P^1 \times \P^1$ with branch locus in anticanonical divisor.
This deformation class of varieties has number $1$ in table $3$ of \cite{MM2},
it has degree $12$, $\chi=-8$ and $\rho=3$.
\end{definition}

\begin{corollary}
$I$-series $I_{6,3} = \L_6 G_{Y^{(3)}_{12}}$ is given by the pullback of hypergeometric series from two-dimensional torus
\begin{equation}
I_{6,3} = \sum_{a,b,c \geq 0} \frac{(2a+2b+2c)!}{a!^2 b!^2 c!^2} t^{a+b+c} 
= 1 + 6 t + 90 t^2 + 1860 t^3 + 44730 t^4 + 1172556 t^5 + \dots
\end{equation}
\end{corollary} 

\begin{remark}
Series $I_{6,3}$ and $I_{6,3;2}$ are related by change of coordinates $I_{6,3;2} (t) = I_{6,3} (t^2)$.
\end{remark}

\section{Eta-products, Hauptmoduln and their M-series} \label{modular}


Let $\eta(q)$ be Dedekind's eta-function: $\eta(q) = q^{\frac{1}{24}} \prod_{n \geq 1} (1-q^n)$.

Consider $4$ eta-products and $1$ eta-quotient:
\begin{gather}
\eta_{6+} = \eta(q) \eta(q^2) \eta(q^3) \eta(q^6) \\
\eta_{10+} = \eta(q) \eta(q^2) \eta(q^5) \eta(q^{10}) \\
\eta_{12+} = \frac{\eta(q^2)^4 \eta(q^6)^4}{\eta(q) \eta(q^3) \eta(q^4) \eta(q^{12})} \\
\eta_{14+} = \eta(q) \eta(q^2) \eta(q^7) \eta(q^{14}) \\
\eta_{15+} = \eta(q) \eta(q^3) \eta(q^5) \eta(q^{15})
\end{gather}

Let $\sigma_1(n)$ be $-24$ times valuation of $\eta_{n+}$.

For a constant $c$ and a conjugacy class $g$ of Monster simple group 
denote by $T_{g,c}$ its McKay-Thompson series (see \cite{CN})
with constant term normalized to be $c$: $T_{g,c} = \frac{1}{q} + c + \sum_{n \geq 1} a_i(g) q^n$.

Take
$H_{6A,2} = T_{6A,10}$,
$H_{6A,3} = T_{6A,14}$,
$H_{10A} = T_{10A,4}$,
$H_{12A} = T_{12A,6}$,
$H_{14A,s} = T_{14A,s}$,
$H_{15A,s} = T_{15A,s}$:

\begin{gather}
T_{6A,0} = \frac{1}{q} + 79 q + 352 q^2 + 1431 q^3 + 4160 q^4 + 13015 q^5 + 31968 q^6 + \dots \\
H_{10A} = 8 + \frac{\eta^4(q)\eta^4(q^5)}{\eta^4(q^2)\eta^4(q^{10})} +
16 \frac{\eta^4(q^2)\eta^4(q^{10})}{\eta^4(q)\eta^4(q^5)} 
= \frac{1}{q} + 4 + 22 q + 56 q^2 + 177 q^3 + 352 q^4 + \dots \\
H_{12A} = \left(\frac{\eta(q^2)^2 \eta(q^6)^2}{\eta(q) \eta(q^3) \eta(q^4) \eta(q^{12})}\right)^6 =
 \frac{1}{q} + 6 + 15 q + 32 q^2 + 87 q^3 + 192 q^4 + \dots \\
 H_{14A} = 4 + \frac{\eta^3(q)\eta^3(q^7)}{\eta^3(q^2)\eta^3(q^{14})} +
8 \frac{\eta^3(q^2)\eta^3(q^{14})}{\eta^3(q)\eta^3(q^7)}
= \frac{1}{q} + 1 + 11 q + 20 q^2 + 57 q^3 + 92 q^4 + \dots \\
 H_{15A} = 3 + \frac{\eta^2(q)\eta^2(q^5)}{\eta^2(q^3)\eta^2(q^{15})} + 
9 \frac{\eta^2(q^3)\eta^2(q^{15})}{\eta^2(q)\eta^2(q^5)} 
= \frac{1}{q} + 1 + 8 q + 22 q^2 + 42 q^3 + 70 q^4 + \dots
\end{gather}

Define $M_n(t)$ as power-series satisfying the functional equation
$M_n(\frac{1}{H_n(q)}) = \eta_{n+} \cdot H_n^\frac{\sigma_1(n)}{24}$:

\begin{gather}
M_{6,2}(H_{6A,2}^{-1}(q)) = \eta_{6+} \cdot H_{6A,2}^\frac{1}{2} \\
M_{6,3}(H_{6A,3}^{-1}(q)) = \eta_{6+} \cdot H_{6A,3}^\frac{1}{2} \\
M_{10}(H_{10}^{-1}(q)) = \eta_{10+} \cdot H_{10A}^{\frac{3}{4}} \\
M_{12}(H_{12}^{-1}(q)) = \eta_{12+} \cdot H_{12A}^{\frac{1}{2}} \\
M_{14,s}(H_{14,s}^{-1}(q)) = \eta_{14+} \cdot H_{14A,s} \\
M_{15,s}(H_{15,s}^{-1}(q)) = \eta_{15+} \cdot H_{15A,s}
\end{gather}

\section{Differential equations and their solutions} \label{ode}

Let $t$ be a coordinate on $G_m = \Spec \C[t,t^{-1}]$ and $D = t\frac{d}{dt}$.

\begin{definition}[\cite{Go1}] \label{myd3}
\emph{Normalized operators of type $D3$} is the following $5$-dimensional family
of differential operators depending on parameters $b_1,b_2,b_3,b_4,b_5$
\footnote{Original definition has the other basis 
$a_{01},a_{02},a_{03},a_{11},a_{12}$
for parameter space $\A^5$.
Bases $a$ and $b$ are equivalent over $\Z$:
$b_1 = a_{11}$, 
$b_2 = a_{12} + 2 a_{01} - a_{11}^2$,
$b_3 = a_{01}$, 
$b_4 = a_{02} - a_{01} a_{11}$,
$b_5 = a_{03} - a_{01}^2$;
$a_{01} = b_3$,
$a_{02} = b_4 + b_1 b_3$,
$a_{03} = b_5 + b_3^2$,
$a_{11} = b_1$,
$a_{12} = b_2 - 2 b_3 + b_1^2$.
}:
\begin{multline*}
L(b_1,b_2,b_3,b_4,b_5) = D^3 -t \cdot b_1 D(D+1)(2D+1) -t^2 \cdot (D+1)(b_2 D(D+2)  + 4 b_3) - \\
-t^3 \cdot b_4 (D+1)(D+2)(2D+3) -t^4 \cdot b_5 (D+1)(D+2)(D+3)
\end{multline*}
\end{definition}

Define $L_{6,2}$, $L_{6,3}$, $L_{10}$, $L_{12}$, $L_{14}$, $L_{15}$ as follows:

\begin{gather} 
\label{ode62}   L_{6,2} = L(6, 368, 88, 1056, 3584) \\
\label{ode63}   L_{6,3} = L(8, 360, 108, 864, 2160) \\
\label{ode10}   L_{10}  = L(2, 112, 28, 184, 336) \\
\label{ode12}   L_{12}  = L(2, 80, 24, 96, 0) \\
\label{ode14}   L_{14}  = L(1, 59, 16, 68, 80) \\
\label{ode15}   L_{15}  = L(1, 43, 12, 78, 216)
\end{gather}


Locally equation $L_n$ has $1$-dimensional space of analytic solutions spanned by $F_n$:
\begin{gather}
\label{f6,2} F_{6,2}(t) = 1 + 44 t^2 + 528 t^3 + 11292 t^4 + 228000 t^5 + 4999040 t^6 + 112654080 t^7 + \dots \\
\label{f6,3} F_{6,3}(t) =  1 + 54 t^2 + 672 t^3 + 15642 t^4 + 336960 t^5 + 7919460 t^6 + 191177280 t^7 + \dots \\
\label{f10} F_{10}(t) = 1 + 14 t^2 + 72 t^3 + 882 t^4 + 8400 t^5 + 95180 t^6 + 1060080 t^7 +\dots  \\
\label{f12} F_{12}(t) = 1 + 12 t^2 + 48 t^3 + 540 t^4 + 4320 t^5 + 42240 t^6 + 403200 t^7 + \dots \\
\label{f14} F_{14}(t) = 1 + 8 t^2 + 24 t^3 + 240 t^4 + 1440 t^5 + 11960 t^6 + 89040 t^7 + \dots \\
\label{f15} F_{15}(t) = 1 + 6 t^2 + 24 t^3 + 162 t^4 + 1080 t^5 + 7620 t^6 + 55440 t^7 + \dots
\end{gather}

\section{Equivalence of realizations} \label{equivalence}

\begin{lemma} \label{lemmai}
$I$-series $I_n$ are solutions to differential equations $L_n$ listed in \ref{ode}, i.e.
$\N I_n = F_n$.
\end{lemma}
\begin{proof}
By \cite{Ga4} 
$G$-function $G_V$ of $G$-Fano threefold $V$ satisfy ODE of order $4$
and its Fourier--Laplace transform $\widehat{G}_V$
satisfy a Fuchsian ODE of order $3$.
\end{proof}

\begin{lemma} \label{lemmam}
$M$-series $M_n$ are solutions to differential equations $L_n$ listed in \ref{ode}, i.e. 
$\N M_n = F_n$.
\end{lemma}
\begin{proof}
By Proposition 21 of \cite{Zag} function $M_n$ satisfy some differential equation
of order $3$ in $D = t \frac{d}{dt}$.
\end{proof}

We have an immediate
\begin{corollary}
For every $n$ series $F_n$, $\N I_n$ and $\N M_n$ coincide.
\end{corollary}

and it implies the main
\begin{theorem} \label{i=m}
For every $n$ there are constants $s_n$ and $c_n$ such that
$I$-series of $G$-Fano threefolds satisfy generalized miraculous eta-product formula \ref{mepf}:
$I_{n,s}(H_{n,c}^{-1}) = \eta_{n+} \cdot H_{n,c}^\frac{\sigma_1(n)}{24}$.
\end{theorem}

\section{Golyshev's modularity of minimal Fano threefolds, and beyond} \label{golmod}

Let $Y_N$ be a Fano threefold with $H^2(Y,\Z) = \Z K_Y$ and $(-K_Y)^3 = 2N$.
Let $G_Y = 1 + \sum_{n \geq 2} G^{(n)}(Y) t^n$ be its $G$-series.
Golyshev's modularity conjecture states that for every $N = 1,\dots,9, 11$ 
there exists such a constant $s_N$,
and a Monster conjugacy class $g_N$ 
($N+N$ in notations of \cite{CN}), 
and a constant $c_N$ such that
$$ \eta^2(q) \eta^2(q^N) \cdot T_{g_N,c_N}^{\frac{N+1}{12}}(q) = I_{Y_N,s_N}(\frac{1}{T_{g_N,c_N}(q)}) $$
where $T_{g_N,c_N} = \frac{1}{q} + c_N + O(q)$ is McKay-Thompson series for conjugacy class $g_N$
with constant term $c_N$.

For $N=6$ the conjugacy class $6+6$ is $6B$, for other values of $N$ it is $NA$.

In the table we specify values $N$, $g = g_N$, $c = c_N$ and $s = s_N$
for $16$ $G$-Fano threefolds of index $1$ --- $10$ Golyshev's cases and $6$ other
$G$-Fano threefolds with $N=6,6,10,12,14,15$ that are explained in this paper.

For integer $N$ we define numbers:

$\phi(N) = N \prod_{p | N} (1 - p^{-1})$,
$\psi(N) = N \prod_{p | N} (1 + p^{-1})$,
$\epsilon(N) = \frac{24}{\psi(N)}$,
$\iota(N) = \sum_{M|N} \phi(M) \epsilon(M)$.

\begin{tabular}{|l|c|c|c|c|c|c|c|c|c|c|c|c|c|c|c|c|}
\hline
N & 1  & 2  & 3  & 4  & 5  & 6  & 6  & 6  & 7  & 8  & 9  & 10  & 11  & 12  & 14  & 15    \\
\hline
$\epsilon(N)$ & $24$ & $8$ & $6$ & $4$ & $4$ & $2$ & $2$ & $2$ & $3$ & $2$ & $2$ & $4/3$ & $2$ & $1$ & $1$ & $1$ \\
\hline
$\iota(N)$    & $24$ &$16$ &$12$ &$10$ & $8$ & $8$ & $8$ & $8$ & $6$ & $6$ & $4$ & $8*$ & $4$ & $5*$ & $4$ & $4$  \\
\hline
s & 120& 24 & 12 &  8 &  6 &  6 &  5 &  4 &  4 &  4 &  3 &  2  &   * &  4  &   * &   *   \\
\hline
c & 744& 104& 42 & 24 & 16 & 14 & 12 & 10 &  9 &  8 &  6 &  4  & s+2 &  6  & s+1 & s+1   \\
\hline
g & 1A & 2A & 3A & 4A & 5A & 6A & 6B & 6A & 7A & 8A & 9A & 10A & 11A & 12A & 14A & 15A   \\
\hline
$\rho$&1& 1 &  1 &  1 &  1 &  3 &  1 &  2 &  1 &  1 &  1 &  2  &   1 &  4  &   2 &   3   \\
\hline
\end{tabular}

\begin{proposition}
Number $\psi(N)$ equals to index of $\Gamma_0(N)$ in $SL(2,\Z)$.
Number $\phi(N)$ equals to index of $\Gamma_1(N)$ in $\Gamma_0(N)$.
\end{proposition}

Number $\epsilon(N)$ is integer only for $15$ values of $N$.
\begin{proposition}
Let $N$ be one of $1,\dots,8,11,14,15,23$.
Denote by $M_{23}$ the Mathieu group and by $V$ its natural $24$-dimensional representation
induced from Mathieu group $M_{24}$ (which in turn induced from Conway group $Co_0$).
Let $g \in M_{23}$ be an element of Mathieu group $M_{23}$ of order $N$.
Then number $\epsilon(N)$ equals to the trace $Tr {g : V}$
and 
number $\iota(N)$ equals to the dimension of invariants $\dim V^g$.
\end{proposition}

\begin{remark}
Note that for $N = 11, 14, 15$ the respective space of modular forms is $2$-dimensional
and any choice of $s$ produces a modular relation.
There is a particular choice of $c$ depending on $s$: the difference $(c-s)$ is an invariant of
Fano threefold.
For $N \neq 11, 14, 15$ the choice of $s$ is unique: $s$ is a natural number such that
$I$-function has singularity at $0$\footnote{There is a single ambiguity in case $N=7$:
one have to choose $s=4$ but not $s=5$.}.
\end{remark}


\section{Mathieu groups.}
\begin{definition}
Let $S$ be the set of $24$ points and $S_{24} = Aut(S)$ is its group of automorphisms.
Let $M = S^\Q$ be a vector space of the tautological $24$-dimensional representation of $S_{24}$.
Mathieu group $M_{24}$ is a particular simple subgroup of $S_{24}$
of order $244823040 = 23 \cdot 11 \cdot 7 \cdot 5 \cdot 3^3 \cdot 2^{10} = 24 \cdot 23 \cdot 22 \cdot 21 \cdot 20 \cdot (16 \cdot 3)$.
\end{definition}
Natural action $M_{24} : S$ is $5$-transitive.

\begin{definition}
Stabilizer of one point in this action is simple Mathieu group $M_{23}$
of order $23 \cdot 11 \cdot 7 \cdot 5 \cdot 3^2 \cdot 2^7$.
\end{definition}

\begin{definition}
Any transposition $h \in S_{24}$ of symmetric group decomposes into the product of cycles,
so we will say \emph{Frame shape} of $h$ is $\prod_{n\geq 1} {\mathbf i}^{a_i}$
where $a_i(h)$ is number of cycles of length $n$. 
Frame shape is a complete invariant for
conjugacy classes of $S_{24}$ (and complete up to power-equivalence for $M_{24}$ and $M_{23}$).
\end{definition}
\begin{example}
\begin{enumerate}
\item
$S_{24}$ has $1575$ Frame shapes and conjugacy classes (equal to number of partitions of $24$),
\item
$M_{24}$ has $21$ Frame shapes and $26$ conjugacy classes,
\item
$M_{23}$ has $12$ Frame shapes and $17$ conjugacy classes.
\end{enumerate}
\end{example}

\begin{definition}
Order of element (conjugacy class, Frame shape) $h$
 equals to the least common multiple of cycle lengths in its cycle decomposition: $n(h) = lcm(a_1,\dots,a_{24})$.
Denote 
by $G(h) = Tr {h : \Q^{24} = H^\bullet(24 points,\Q)}$
the number $a_1$ of cycles of length one 
(it can be computed via Lefschetz fixed point formula on finite set of $24$ elements).
\end{definition}

Obviously $G(h)$ is integer non-negative and if $h \in M_{23}$ then $G(h) \geq 1$.
\begin{proposition}[Frobenius, Mukai]
\begin{enumerate}
\item
$h \in M_{24}$ comes from $M_{23}$ $\iff$ $G(h) \geq 1$.
\item
For element $h \in M_{23}$ number $G(h)$ depends only on $n(h)$: $G(h) = \epsilon(n)$.
\item
There are $12$ orders of elements in $M_{23}$: from $1$ to $8$, $11$, $14$, $15$ and $23$.
Frame shapes of $M_{23}$ are determined by the orders.
\end{enumerate}
\end{proposition}

\begin{proposition}
$M_{23}$ has the following $12$ Frame shapes:

\begin{tabular}{|l|c|c|c|c|c|c|c|c|c|c|c|c|}
\hline
$g$ & $1^{24}$ & $1^8 2^8$ & $1^6 3^6$ & $1^4 2^2 4^4$ & $1^4 5^4$ & $1^2 2^2 3^2 6^2$ & $1^3 7^3$
& $1^2 2^1 4^1 8^2$ & $1^2 11^2$ & $1^1 2^1 7^1 14^1$ & $1^1 3^1 5^1 15^1$ & $1^1 23^1$      \\
\hline
$N$ & $1$ & $2$ & $3$ & $4$ & $5$ & $6$ & $7$ & $8$ & $11$ & $14$ & $15$ & $23$ \\
\hline
$w$ & $12$ & $8$ & $6$ & $5$ & $4$ & $4$ & $3$ & $3$ & $2$ & $2$ & $2$ & $1$ \\
\hline
\end{tabular}

$M_{24}$ has the following $9$ extra Frame shapes:
\begin{tabular}{|l|c|c|c|c|c|c|c|c|c|}
\hline
$g$ & $2^{12}$ & $3^8$ & $2^4 4^4$ & $4^6$ & $6^4$ & $2^2 10^2$ & $2^1 4^1 6^1 12^1$ & $12^2$ & $3^1 21^1$ \\
\hline
$n$ & $2$ & $3$ & $4$ & $4$ & $6$ & $10$ & $12$ & $12$ & $21$ \\
\hline
$N$ & $4$ & $9$ & $8$ & $16$ & $36$ & $20$ & $24$ & $144$ & $63$ \\
\hline
$w$ & $6$ & $4$ & $4$ & $3$ & $2$ & $2$ & $2$ & $1$ & $1$ \\ 
\hline
\end{tabular}
\end{proposition}

\section{Mason's cusp-forms.}
Given a Frame shape $g = \prod {\mathbf i}^{a_i}$
consider a function $\eta_g = \prod \eta(q^n)^{a_i}$,
where $\eta(q) = q^{\frac{1}{24}} \prod_{m \geq 1} (1-q^m)$ is Dedekind's eta-function.

Define \emph{weight} of Frame shape as $w(g) = \frac{\sum a_i}{2}$
and \emph{level} as $N(g) = gcd(a_1,\dots,a_{24}) \cdot lcm(a_1,\dots,a_{24})$.
\footnote{In our examples it will be $\min\{i | a_i \neq 0\} \cdot \max\{i | a_i \neq 0\}$.}

\begin{theorem}[Mason \cite{Mas}]
Let $g$ be one of $21$ Frame shapes of $M_{24}$.
Then $\eta_g$ is a cusp-form and Hecke-eigenform
of weight $w(g)$ and level $N(g)$ with quadratic nebentypus character
 (if weight $w(g)$ is even then the character is trivial).
Moreover, all these functions $\eta_g$ form a character of a particular
graded representation of $M_{24}$ functorially constructed from $M$.
\end{theorem}

\begin{theorem}[Dummit,Kisilevsky,McKay \cite{DKM};Koike \cite{Koi1}; Mason \cite{Mas}]
Only for $30$ out of $1575$ Frame shapes $g$ of $S_{24}$
the respective eta-product $\eta_g$ is a Hecke eigen-cuspform.
There are $2$ extra Frame shapes with non-integer weight ($24^1$ and $8^3$)
and $7$ extra Frame shapes with integer weight:
\begin{tabular}{|l|c|c|c|c|c|c|c|}
\hline
$g$ & $3^2 9^2$ & $4^2 8^2$ & $2^3 6^3$ & $2^1 22^1$ & $4^1 20^1$ & $6^1 18^1$ & $8^1 16^1$ \\
\hline
$n$ &       $9$ &       $8$ &       $6$ &       $22$ &       $20$ &       $18$ &       $16$ \\
\hline
$N$ &      $27$ &      $32$ &      $12$ &       $44$ &       $80$ &      $108$ &      $128$ \\
\hline
$w$ &       $2$ &       $2$ &       $3$ &        $1$ &        $1$ &        $1$ &        $1$ \\
\hline
\end{tabular}
It is known all of them come as characters of extension $2^{11} \dot M_{24}$.
\end{theorem}

\section{Symplectic automorphisms of K3 surfaces.}
\begin{theorem}[Nikulin] \label{thm-nikulin}
Let $g$ be an automorphism of finite order $N$ on $K3$ surface $S$
 preserving the holomorphic volume form $\omega$: $g^N = 1$, $g^* \omega = \omega$.
Denote by $F(g)$ the number of its fixed points: $F(g) = Tr {g : H^\bullet(S,\Q)}$,
number $F(g)$ can be computed by combining topological and holomorphic Lefschetz fixed point formulae on $K3$ surface.
Then
\begin{enumerate}
\item
Order of symplectic automorphism is bounded by $N \leq 8$.
\item
$F(g)$ depends only on the order and $F(g) = \epsilon(N)$.
\end{enumerate}
\end{theorem}

\begin{theorem}[Mukai] \label{thm-mukai}
Finite group $G$ acts on $K3$ surface $S$ preserving the holomorphic volume form $\omega$
$\iff$
the following two conditions are satisfied:
\begin{enumerate}
\item
$G$ can be embedded in $M_{23}$
\item
tautological action of $G$ on set $S$ has at least $5$ orbits,
or in other words --- the dimension of invariants of $G$-representation $H^\bullet(S,\Q)$
is at least $5$.
\end{enumerate}
\end{theorem}

\begin{problem}
Is there any direct geometric relation between 
$G$-Fano threefolds and symplectic automorphisms of $K3$ surfaces?
\end{problem}

\section{Measuring rationality}
\begin{definition}
We say that deformation class of smooth Fano threefolds is of \emph{irrational} type if there is at least one
irrational (over $\C$) variety in this family. Otherwise we say it is of \emph{rational type}.
\end{definition}

\begin{proposition} \label{grat}
All $G$-Fano threefolds of higher rank are rational.
\end{proposition}
\begin{proof}
One can check it case by case.
Two varieties of index Fano index $2$ ($\P^1 \times \P^1 \times \P^1$ and variety of complete flags $W$)
are obviously rational.
For six varieties of index one:
\begin{enumerate}
\item
Projection from $Y_{15} \subset \P^2 \times \P^2 \times \P^2$ to $\P^1 \times \P^2$ (contract the first $\P^2$
and project from a point on the second $\P^2$ factor) is birational.
Also $Y_{15}$ is known to be blowup of $W$.
\item
Projection from $Y_{14} \subset Q \times Q$ to one of the factors is birational (inverse to the blowup of a twisted quartic).
\item
Projection from $Y_{12} = X_{(1,1,1,1)} \subset \P^1 \times \P^1 \times \P^1 \times \P^1$ to $\P^1 \times \P^1 \times \P^1$
is birational. In fact it is a blowup with center in elliptic curve, which is the
base locus of a pencil in linear system $\O_{\P^1\times\P^1\times\P^1}(1,1,1)$.
\item
Projection from $Y_{10} = X_{(1,1),(1,1),(1,1)} \subset \P^3 \times \P^3$ to one of the factors is birational.
\item
Varieties $Y_{6,2}$ are divisors of type $(2,2)$ in $\P^2 \times \P^2$.
Projection to the first factor is a conic bundle over $\P^2$ with degeneracy locus of degree six.
By Iskovskikh criterium they are rational.
\item
Varieties $Y_{6,3}$ are covers of $\P^1 \times \P^1 \times \P^1$ branched in anticanonical divisor.
Composition of the double cover and projection to the product of first two factors
gives a conic bundle over $\P^1 \times \P^1$.
\end{enumerate}
\end{proof}

\begin{proposition}
Let $X_N$ be a deformation class $G$-Fano threefolds of index $r=1$ and degree $2 N$.
Then $X_N$ is of rational type $\iff \epsilon(N) \geq 2$.
\end{proposition}
\begin{proof}
Case by case. Combine proposition \ref{grat} and tables in the end of \cite{IP}.
\end{proof}

\bigskip

This work was done in $2010$ and I reported it in some talks:
\begin{itemize}
\item
``(Mirror) modularity of Fano threefolds'' on 2010.08.24 at Tokyo Metropolitan University,
\item
``G-Fano threefolds and Mathieu group'' on 2011.01.31 at University of Miami,
\item
``Fano and Mathieu'' on 2011.03.28 at the joint seminar of the Sector 4.1 (IITP RAS) and Poncelet French-Russian laboratory ``Arithmetic Geometry and Coding Theory'' in Moscow,
\item
``Fano and Mathieu'' on 2011.05.12 at \emph{Number Theory and Physics at the Crossroads} workshop in Banff International Research Station,
\item
``New ways to count up to $15$'' on 2011.06.03 at Algebraic Geometry Seminar in Kyoto University.
\end{itemize}

{\bf Acknowledgements.}
Author thanks Tom Coates, Alessio Corti, Vasily Golyshev, Shinobu Hosono, Ludmil Katzarkov, Bumsig Kim, Yuri Prokhorov,
Duco van Straten, Todor Milanov, Kyoji Saito, Maria Vlasenko, Noriko Yui
for useful discussions.

{\bf Funding.}
This work was supported by 
World Premier International Research Center Initiative (WPI Initiative), MEXT, Japan,
Grant-in-Aid for Scientific Research (10554503) from Japan Society for Promotion of Science,
Grant of Leading Scientific Schools (N.Sh. 1987.2008.1 and 4713.2010.1)
This is an update for the second part of my earlier SFB45 preprint \cite{Ga3},
supported by INTAS 05-100000-8118.


\end{document}